\documentclass[12pt,a4paper]{article}
\usepackage{amsmath,amsthm,amssymb}
\usepackage{mathtools}
\usepackage[margin=1in]{geometry}
\usepackage{wrapfig}
\usepackage{graphicx}
\usepackage{url}
\usepackage{hyperref}

\newtheorem{proposition}{Proposition}

\newtheorem{theorem}{Theorem}
\newtheorem{question}{Question}

\begin{document}
\title{Bounding Klarner's constant from above using a simple recurrence}
\author{Vuong Bui\thanks{LIRMM, Universit\'e de Montpellier, CNRS, 161 Rue Ada, 34095 Montpellier, France and UET, Vietnam National University, Hanoi, 144 Xuan Thuy Street, Hanoi 100000, Vietnam  (\href{mailto://bui.vuong@yandex.ru}{\texttt{bui.vuong@yandex.ru}})\\
Part of the work was supported by the Deutsche Forschungsgemeinschaft
(DFG) Graduiertenkolleg ``Facets of Complexity'' (GRK 2434).}}

\date{}
\maketitle
\begin{abstract}
Klarner and Rivest showed that the growth of the number of polyominoes, also known as Klarner's constant, is at most $2+2\sqrt{2}<4.83$ by viewing polyominoes as a sequence of twigs with appropriate weights given to each twig and studying the corresponding multivariate generating function. In this short note, we give a simpler proof by a recurrence on an upper bound. In particular, we show that the number of polyominoes with $n$ cells is at most $G(n)$ with $G(0)=G(1)=1$ and for $n\ge 2$,
\[
    G(n) = 2\sum_{m=1}^{n-1} G(m)G(n-1-m).
\]
It should be noted that $G(n)$ has multiple combinatorial interpretations in literature.
\end{abstract}

\section{Introduction}
Polyomino is a popular geometric figure that must have been used either for practical or recreational purposes for a long time, either implicitly or explicitly. Rigorously speaking, a polyomino is an edge-connected set of cells on the square lattice.
A nice introduction and survey of polyominoes can be found in the chapter ``Polyominoes'' of ``Handbook of Discrete and Computational Geometry'' \cite{barequet2017polyominoes}.
We work with fixed polyominoes in this article, in the sense that two polyominoes are equivalent if one is a translate of the other. The number of polyominoes with $n$ cells can be found in the sequence $A001168$ of The On-Line Encyclopedia of Integer Sequences \cite{oeis}. The study of the asymptotic growth of polyominoes starts with Klarner's paper \cite{klarner1967cell}, which observes that the number $A(n)$ of (fixed) polyominoes with $n$ cells is supermultiplicative. Together with the boundedness of $\sqrt[n]{A(n)}$ (by $A(n)\le \binom{3n}{n-1}$, see \cite{eden1961two}), we have a corollary that the growth constant is actually the limit
\[
    \lambda \coloneqq \lim_{n\to\infty} \sqrt[n]{A(n)},
\]
which is also called Klarner's constant.
Another corollary is the lower bound that for every $n$,
\[
    \lambda\ge\sqrt[n]{A(n)}.
\]
In particular, with the recent advance 
\[
    A(70)=18,500,792,645,885,711,270,652,890,811,942,343,400,814
\]
in \cite{barequet2024counting},
we have $\lambda\ge \sqrt[70]{A(70)} > 3.76049$. 
However, the state of the art for the lower bound is $\lambda > 4.0025$, by a related notion of twisted cylinders \cite{barequet2016lambda}. The lower bound is quite close to the (unproved) estimate $4.0625696\pm 0.0000005$ by Jensen \cite{jensen2003counting}.

It turns out that obtaining upper bounds on $\lambda$ seems to be harder, and the bounds are in general far from the believed value. The first bound is due to Eden \cite{eden1961two} by an enumeration so that each polyomino can be seen as sequence of twigs:
\[
    A(n)\le \binom{3n}{n-1},
\]
which is followed by
\[
    \lambda\le \frac{27}{4}=6.75.
\]
Klarner and Rivest \cite{klarner1973procedure} optimized the approach by enumerating polyominoes with a set of $5$ twigs, reduced from $8$ twigs as in Eden's implementation, so that we obtain immediately the bound $\lambda\le 5$. This bound can be improved by assigning appropriate weights to the twigs and studying the corresponding multivariate generating function, by which they obtain
\[
    \lambda\le 2+2\sqrt{2}<4.83.
\]
Also in the same article, Klarner and Rivest extended the approach with more than $2$ million twigs to obtain
\[
    \lambda \le 4.649551.
\]
Using the power of computation nowadays with thousands of billion twigs and some additional tricks, Barequet and Shalah \cite{barequet2022improved} obtained
\[
    \lambda \le 4.5252.
\]

In general, it seems to be very hard to obtain a good upper bound without using huge computational power, even if we optimize certain steps. Therefore, it may be interesting to look at the best upper bound that we can prove manually without extensive computation: $\lambda\le 2+2\sqrt{2}$. Although the treatment in \cite{klarner1973procedure} gives a general framework, we will prove the bound using a simpler and more straightforward way, by a recurrence of an upper bound on $A(n)$. This allows a simple investigation of an ordinary generating function, as opposed to multivariate generating functions in \cite{klarner1973procedure}. Usually the approach is to bound the quantity we need to study by another quantity which is more well behaved, e.g., it satisfies some recurrence. This is a popular technique, for example, one can find an instance in \cite{barequet2019improved}, or a more general application of the technique in \cite{barequet2021concatenation}. A merit of this work is that instead of working directly on the original quantity $A(n)$, we study the number of pairs of a polyomino and a cell so that they satisfy some certain properties, which involve the neighbors of the cell in the polyomino. On the one hand, although the number of such pairs may be linearly many times larger than the original quantity, both still grow at the same rate. On the other hand, by considering the neighbors, we may have more useful recurrences. In the next section, we show that $A(n)$ is at most $G(n)$, which is defined by $G(0)=G(1)=1$ and for $n\ge 2$,
\[
    G(n) = 2\sum_{m=1}^{n-1} G(m)G(n-1-m).
\]
In fact, $G(n)$ for $n=1,2,\dots$ is the sequence $A071356$ of The On-Line Encyclopedia of Integer Sequences \cite{oeis}, which has multiple interpretations. For example, $G(n)$ is the number of Motzkin's paths from $(0,0)$ to $(n,0)$ where the level and up steps are bicolored \cite{shapiro2009bijection}.
The sequence can be related to other convolutions of similar types, a familiar one being Catalan's number.
The proof that $A(n)\le G(n)$ involves decomposing a polyomino into smaller parts that still guarantee some properties.
It leads to the following natural question.
\begin{question}
    Is there a simple injection from polyominoes to Motzkin's paths mentioned above (or any other equivalent combinatorial interpretation of $G(n)$)?
\end{question}

\section{The recurrence}
We prove that $\lambda\le 2+2\sqrt{2}$ by considering pairs $(P,c)$ where $P$ is a polyomino of $n$ cells and $c$ is one of its cells. 

\begin{figure}[h]
    \centering
    \includegraphics[width=0.44\textwidth]{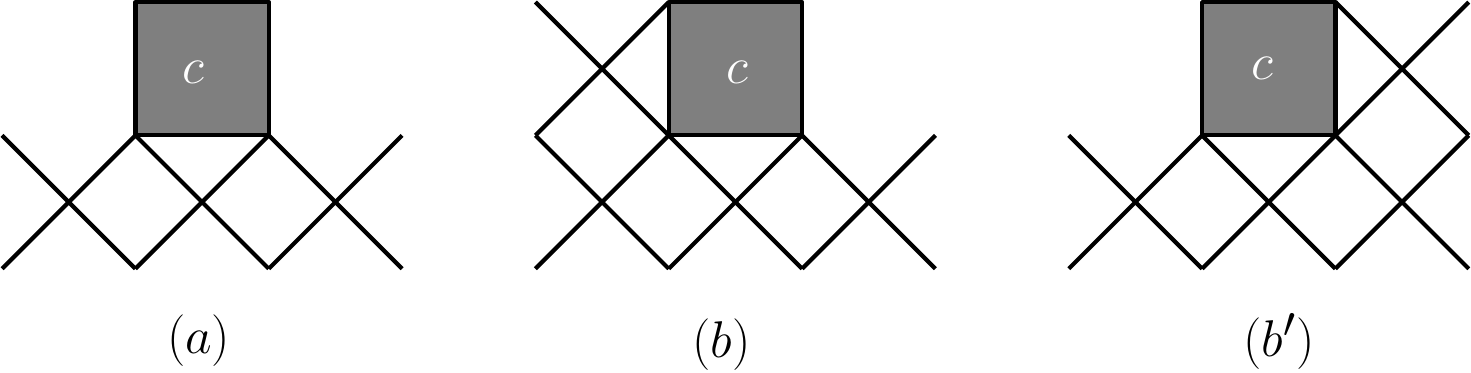}
    \caption{Two types of cells}
    \label{fig:definitions}
\end{figure}

Let $f(n)$ be the number of pairs such that there is no cell in the row below the marked cell $c$ that lies in the same column or an adjacent column, i.e., there is no cell at crossed positions if the marked cell is the black cell in Fig. \ref{fig:definitions}(a). We say that such a pair has Type $(A)$.

Let $g(n)$ be the number of pairs with the same forbidden cells together with the forbidden cell to the left of the marked cell, as in Fig. \ref{fig:definitions}(b). We say that such a pair has Type $(B)$.

Note that $f(n)$ and $g(n)$ are still the same if we rotate or flip the forbidden cells around the marked cells.
For example, if we let the additional forbidden cell be the cell to the right of the marked cell instead, as in Fig. \ref{fig:definitions}(b'), the number of pairs is still $g(n)$.

Before bounding $f(n)$ and $g(n)$ from above, we note that although $f(n)$ and $g(n)$ are both upper bounds on the number $A(n)$ of polyominoes of $n$ cells by considering pairs $(P,c)$ where $c$ is the leftmost cell in the bottommost row of $P$.
Also, all these three numbers have the same growth constant since 
\[
    A(n)\le f(n)\le nA(n),\qquad A(n)\le g(n)\le nA(n).
\]

Let $F(n)$ and $G(n)$ be defined by: $F(0)=F(1)=G(0)=G(1)=1$ and for $n\ge 2$,
\begin{align*}
F(n) &= G(n) + \sum_{\ell,m\ge 1,\,\ell+m=n} G(\ell) G(m),\\
G(n) &= F(n-1)+G(n-1)+\sum_{\ell,m\ge 1,\, \ell+m=n-1} G(\ell) G(m).
\end{align*}

We obtain upper bounds on $f(n)$ and $g(n)$ as follows.

\begin{theorem}
Setting $f(0)=g(0)=1$, we have
$f(n)\le F(n)$ and $g(n) \le G(n)$ for all $n$.
\end{theorem}

\begin{proof}
The relations are trivial for $n=0,1$. Suppose that they are true up to $n-1$, the same will be proved for $n$ as follows.

\begin{itemize}
\item 
The number of pairs $(P,c)$ of Type $(B)$ with no cell to the right of $c$ is at most $f(n-1)\le F(n-1)$, see Fig. \ref{fig:B-expansions}(a). The number of pairs with no cell immediately above $c$ is at most $g(n-1)\le G(n-1)$, see Fig. \ref{fig:B-expansions}(b). If there are two cells, one cell $d$ to the right and the other $e$ immediately above $c$, we can split $P$ after excluding $c$ into two parts $P_1$ and $P_2$ of some $\ell$ and $m$ cells ($\ell,m\ge 1,\, \ell+m=n-1$) such that $(P_1,d)$ and $(P_2,e)$ are both of Type $(B)$ (after rotating/flipping appropriately), see Fig. \ref{fig:B-expansions}(c). Hence,
\begin{equation}\label{ineq:g-G}
    g(n) \le F(n-1)+G(n-1)+\sum_{\ell,m\ge 1,\, \ell+m=n-1} G(\ell) g(m) = G(n).
\end{equation}

\begin{figure}[h]
    \centering
    \includegraphics[width=0.50\textwidth]{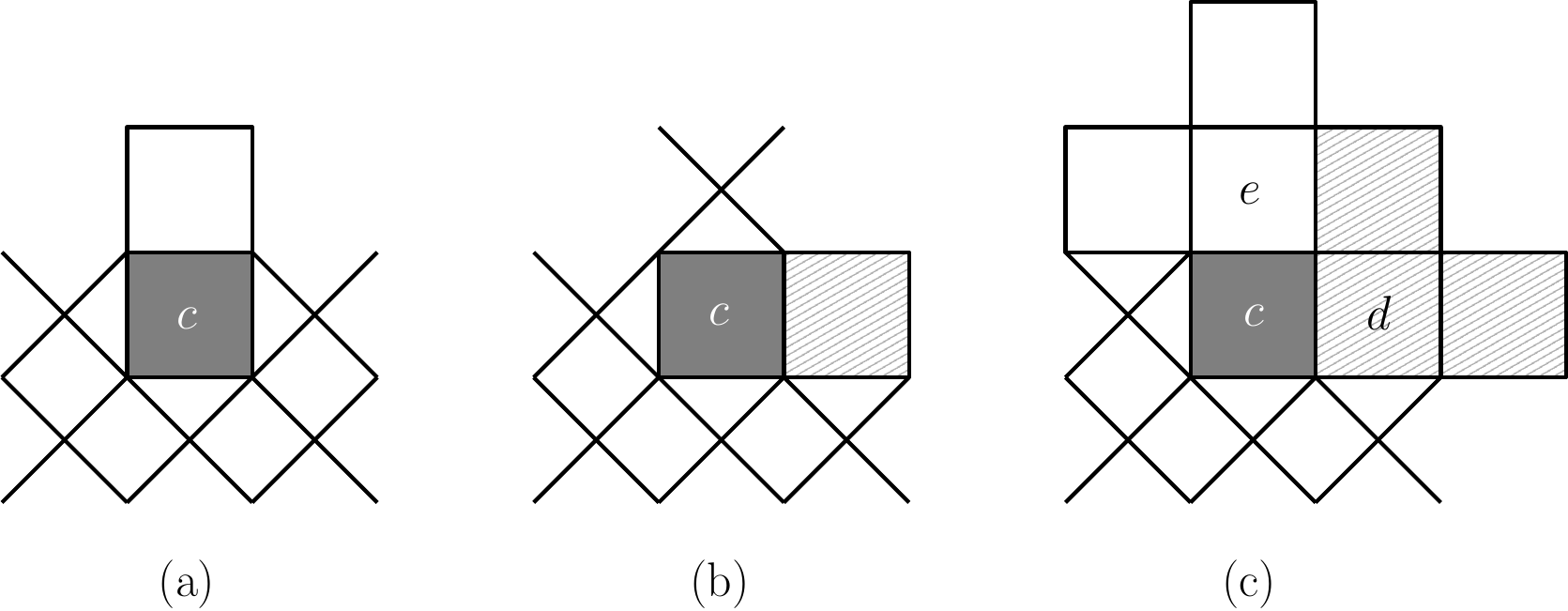}
    \caption{Neighbors for cells of Type (B)}
    \label{fig:B-expansions}
\end{figure}

\item
The number of pairs $(P,c)$ of Type $(A)$ with no cell to the left of $c$ is at most $g(n)\le G(n)$, see Fig. \ref{fig:A-expansions}(a). (Note that $g(n)\le G(n)$ is not due to the induction hypothesis but \eqref{ineq:g-G}.) If there is a cell $d$ to the left of $c$, we can split $P$ into two parts $P_1$ and $P_2$ of some $\ell$ and $m$ cells ($\ell,m\ge 1,\, \ell+m=n$) such that $(P_1,c)$ and $(P_2,d)$ are both of Type $(B)$ (after rotating/flipping appropriately), see Fig. \ref{fig:A-expansions}(b). Hence,
\[
    f(n) \le G(n) + \sum_{\ell,m\ge 1,\,\ell+m=n} G(\ell) G(m) = F(n).
\]

\begin{figure}[h]
    \centering
    \includegraphics[width=0.32\textwidth]{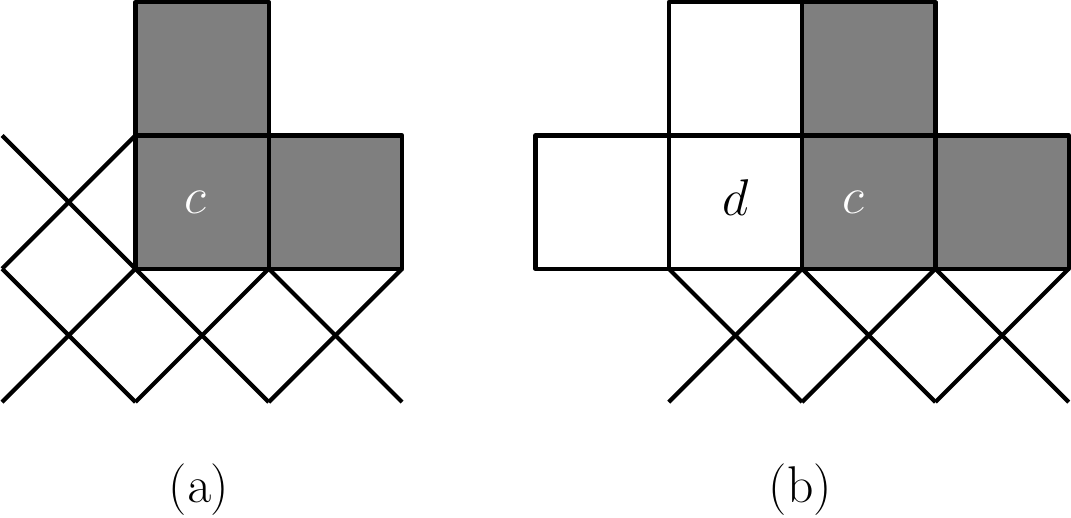}
    \caption{Neighbors for cells of Type (A)}
    \label{fig:A-expansions}
\end{figure}
\end{itemize}
The conclusion is therefore valid for every $n$ by induction.
\end{proof}

In fact, $G(n)$ can be written so that it depends on itself as follows.
\begin{proposition}
For $n\ge 2$,
\[
    G(n) = 2\sum_{m=1}^{n-1} G(m)G(n-1-m).
\]
\end{proposition}
\begin{proof}
Indeed,
\begin{align*}
    G(n)&= F(n-1)+G(n-1)+\sum_{\ell,m\ge 1,\, \ell+m=n-1} G(\ell) G(m)\\
    &= G(n-1) + \sum_{\ell,m\ge 1,\,\ell+m=n-1} G(\ell) G(m) + G(n-1) + \sum_{\ell,m\ge 1,\,\ell+m=n-1} G(\ell) G(m) \\
    &= 2\left(G(n-1) + \sum_{\ell,m\ge 1,\,\ell+m=n-1} G(\ell) G(m)\right)\\
    &= 2\sum_{m=1}^{n-1} G(m)G(n-1-m).\qedhere
\end{align*}
\end{proof}

We proceed with calculating the generating function
\[
    \zeta(n)=\sum_{n\ge 0} G(n) x^n.
\]
\begin{proposition}
\[
    \zeta(x)=\frac{2x+1 - \sqrt{1-4x-4x^2}}{4x}.
\]
\end{proposition}
\begin{proof}
For $n\ge 2$,
\[
    G(n)=\left(2\sum_{m=0}^{n-1} G(m)G(n-1-m)\right) - 2G(n-1),
\]
therefore,
\begin{align*}
\zeta(x)&=1+x+\sum_{n\ge 2} G(x)x^n\\
&=1+x+\left(\sum_{n\ge 2} \sum_{m=0}^{n-1} 2G(m)G(n-1-m)x^n\right) - \sum_{n\ge 2}2 G(n-1)x^n\\
&=1+x+2x([\zeta(x)]^2 - 1)-2x(\zeta(x)-1)\\
&=1+x+2x[\zeta(x)]^2-2x\zeta(x).
\end{align*}
Out of the two solutions of the equation
\[
2x[\zeta(x)]^2 - (2x+1)\zeta(x)+(x+1)=0
\]
of $\zeta(x)$, we choose the solution in the conclusion since it is the one that gives $\lim_{x\to 0} \zeta(x) = 1$. (The other one $\frac{2x+1+\sqrt{1-4x-4x^2}}{4x}$ diverges as $x\to 0$.)
\end{proof}

Since the condition for the discriminant $1-4x-4x^2$ to be nonnegative is $(-\sqrt{2}-1)/2\le x\le (\sqrt{2}-1)/2$, the radius of convergence is $(\sqrt{2}-1)/2$. The growth constant of $G(n)$ is therefore $2/(\sqrt{2}-1)=2+2\sqrt{2}$, which in turn is also an upper bound on Klarner's constant. 

Although we are mainly interested in a simpler proof of the result of Klarner and Rivest in this article, we close the article with the following question.
\begin{question}
    Can we turn the approach into a more general framework as in the work of Klarner and Rivest?
\end{question}

\bibliographystyle{unsrt}
\bibliography{recurpoly}

\end{document}